\documentclass{amsart}
\usepackage{amsmath,amsfonts, amssymb, amscd, graphicx, latexsym, latexsym, amsthm, rotating}
\usepackage[latin5]{inputenc}
\usepackage[small,nohug,heads=littlevee]{diagrams}
\diagramstyle[labelstyle=\scriptstyle]

\newcommand{\Qq}{\mathbb Q}
\newcommand{\Zz}{\mathbb Z}

\newcommand{\des}{{\rm Des}}

\newcommand{\ra}{\rightarrow}

\def\ker{{\rm Ker}}
\def\obs{{\rm Obs}}

\def\res{{\rm Res}}

\def\infl{{\rm Inf}}

\newtheorem{thm}{Theorem}[section]
\newtheorem{pro}[thm]{Proposition}
\newtheorem{cor}[thm]{Corollary}
\newtheorem{lem}[thm]{Lemma}
\newtheorem{rem}{Remark}

\title[Borel-Smith functions and endo-trivial modules]{Gluing Borel-Smith functions and the group of endo-trivial modules}
\author{Olcay COŞKUN}
\thanks{The author is supported by Tübitak through Kariyer
Programı, (Tübitak/3501/109T661).}

\begin{document}

\maketitle
\begin{center}
\footnotesize Boğaziçi Üniversitesi, Matematik Bölümü
80815 Bebek, İstanbul, Turkey.\\
olcay.coskun@boun.edu.tr
\end{center}
\begin{abstract}
The aim of this paper is to describe the group of endo-trivial
modules for a $p$-group $P$, in terms of the obstruction group for
the gluing problem of Borel-Smith functions. Explicitly, we shall
prove that there is a split exact sequence
\begin{displaymath}
\begin{diagram}
0&\rTo&\Zz &\rTo &T(P)&\rTo& \obs(C_b(P))&\rTo&0
\end{diagram}
\end{displaymath}
of abelian groups where $T(P)$ is the endo-trivial group of $P$,
and $C_b(P)$ is the group of Borel-Smith functions on $P$. As a
consequence, we obtain a set of generators of the group $T(P)$
that coincides with the relative syzygies found by Alperin. In
order to prove the result, we solve gluing problems for the
functor $B^*$ of super class functions, the functor $\mathcal
R_\Qq^*$ of rational class functions and the functor $C_b$ of
Borel-Smith functions.

\end{abstract}
\tableofcontents
\section{Introduction}
In \cite{Dade-endo-1} and \cite{Dade-endo-2} Dade introduced the
class of endo-permutation modules for $p$-groups together with a
certain sub-class of endo-trivial modules and classified them for
abelian $p$-groups. He also raised the question of the
classification of endo-permutation modules for all $p$-groups. A
systematic study of the endo-permutation modules is done by
considering the group of the so-called capped endo-permutation
modules, now called the Dade group and denoted by $D(P)$ for the
group $P$. The \emph{subgroup} of endo-trivial modules is denoted
by $T(P)$.

The classification problem for endo-permutation modules is solved,
after 25 years since it was raised, in several steps. One of the
most important steps in this direction is the description of the
group $T(P)$ of the endo-trivial modules by Carlson and
Th\'{e}venaz in \cite{CaThClssTorTri} and \cite{CaThClssTri}. The
final classification was done by Bouc in \cite{Bouc-DadeGroup}
where he used the description of the group $T(P)$ and the theory
of biset functors, introduced and developed by him. See
\cite{Bouc-book} for the theory of biset functors.

This paper is a contribution to the description of the group
$T(P)$. Regarding the original description, in
\cite{Alperin-endo}, Alperin constructed a certain set of relative
syzygies associated to a finite $p$-group $P$ and proveed that
they generate a subgroup of finite index in the group $T(P)$ of
all endo-trivial modules. Later Carlson and Th\'{e}venaz in
\cite{CaThClssTri} showed that this subgroup is equal to the
torsion free part of the group $T(P)$. Also in
\cite{CaThClssTorTri}, they describe the torsion part of the group
$T(P)$ completing the classification of endo-trivial modules for
$p$-groups. The above set of generators is also constructed by
Carlson in \cite{Ca1} using cohomology. We refer to \cite{Ca1},
\cite{CaThClssTri} or \cite{Thevenaz-guided}, for further details.

The aim of this paper is to give an another construction of the
above set of relative syzygies and to prove that they generate the
group $T(P)$. To construct the relative syzygies we relate the
group of endo-trivial modules to the group of obstructions for the
gluing problem for Borel-Smith functions and hence use the theory
of biset functors for the construction. We do use the
classification theorem of Carlson and Th\'{e}venaz regarding the
torsion part of the group $T(P)$ but we do not need the full
classification of endo-permutation (or endo-trivial) modules. See
the end of the paper for a revision of the results that we have
used throughout the paper.

Returning to the new construction of the relative syzygies, the
gluing problem for a biset functor $F$ at a finite group $G$ is
the problem of finding an element $f\in F(G)$ associated to a
given sequence $(f_H)_{1< H\le G}$ where $f_H$ is an element in
$F(N_G(H)/H)$, if it exists, see Section \ref{sec:gluingGeneral}
for details. This problem is first considered by Bouc and
Th\'{e}venaz \cite{BoucThe-glue} for the functor of torsion
endo-permutation modules and later by Bouc \cite{Bouc-glue} for
arbitrary endo-permutation modules, both for $p$-groups with $p$
odd. The first problem is solved whereas the latter is incomplete.
In this paper, we consider the gluing problem for three more biset
functors, namely for the functor of super class functions, for the
functor of rational class functions and for the functor of
Borel-Smith functions.

The relevance of the gluing problem of the functor of Borel-Smith
functions to the group $T(P)$ is given by our main result, Theorem
\ref{thm:trivExa}, which proves that there is a split exact
sequence of abelian groups
\begin{displaymath}
\begin{diagram}
0&\rTo&\Zz &\rTo &T(P)&\rTo& \obs(C_b(P))&\rTo&0 \qquad (\ast)
\end{diagram}
\end{displaymath}
where $T(P)$ is the endo-trivial group of $P$, and $C_b(P)$ is the
group of Borel-Smith functions on $P$.

To obtain this exact sequence, we use the exact sequences
\begin{displaymath}
\begin{diagram}
0&\rTo&C_b &\rTo &B^*&\rTo& D^\Omega&\rTo&0
\end{diagram}
\end{displaymath}
and
\begin{displaymath}
\begin{diagram}
0&\rTo&C_b &\rTo &\mathcal R_\Qq^*&\rTo& D_t^\Omega&\rTo&0
\end{diagram}
\end{displaymath}
of biset functors of Theorem 1.2 and Theorem 1.3 of
\cite{Bouc-yalcin}, and the result of Puig which determines the
kernel of the map $D(P)\ra \varprojlim_{1<Q\le P}D(N_P(Q)/Q)$ as
the endo-trivial group $T(P)$. Finally we need the solutions of
the other to gluing problem mentioned above.

Now the exact sequence $(\ast)$ is the kernel-cokernel sequence of
the following commutative diagram
\begin{displaymath}
\begin{diagram}
0&\rTo&C_b(P) &\rTo &B^*(P)&\rTo^{\Psi_P}&D^\Omega(P)&\rTo&0\\
 & &\dTo^{r_P^{C_b}} & & \dTo^{r_P^{B^*}}&&\dTo^{r_P^{D}}&&\\
0&\rTo&\varprojlim_{1<Q\le P}C_b(W_P(Q)) &\rTo
&\varprojlim_{1<Q\le
P}B^*(W_P(Q))&\rTo^{\tilde\Psi_P}&\varprojlim_{1<Q\le
P}D^\Omega(W_P(Q))
\end{diagram}
\end{displaymath}
with exact rows, where we put $W_P(Q):= N_P(Q)/Q$ and the
splitting can be constructed by considering a similar commutative
diagram with first row equal to the exact sequence of Theorem 1.3
in \cite{Bouc-yalcin}. See Section \ref{sec:mainResults} for
further details.

Regarding the gluing problems, the solution for the gluing problem
for super class functions is in Section \ref{sec:gluingSuperClss}.
We prove, as Theorem \ref{thm:exaBurn}, that there is an exact
sequence
\begin{displaymath}
\begin{diagram}
0&\rTo&\Zz &\rTo &B^*(G)&\rTo& \varprojlim_{1<H\le
G}B^*(N_G(H)/H)&\rTo&0
\end{diagram}
\end{displaymath}
of abelian groups. In particular, the gluing problem for super
class functions always has infinitely many solutions. Moreover, it
is clear from the proof of this theorem, given in Section
\ref{sec:gluingSuperClss}, that solutions to a given gluing data
differ only by their evaluations at the trivial subgroup.

The solution of the problem for rational class functions for a
$p$-group $P$ of rank at least $2$ is given by the exact sequence
\begin{displaymath}
\begin{diagram}
0&\rTo&\mathcal R_\Qq^*(P) &\rTo^{r_P} &\varprojlim_{1<Q\le
P}\mathcal R_\Qq^*(N_P(Q)/Q)&\rTo^{\tilde d_P}& \tilde
H^0(\mathcal A_{\ge 2}(P), p\Zz)^P&\rTo&0
\end{diagram}
\end{displaymath}
of abelian groups. See Section \ref{sec:gluingRatclss} for the
details of the notation. Note that if the rank of $P$ is $1$ then
$\mathcal R_\Qq^* = B^*$ and the problem is reduced to the
previous case.

The gluing problem for Borel-Smith functions is solved in Section
\ref{sec:gluingBS} where we prove that for a $p$-group of rank at
least $2$, there is an exact sequence
\begin{displaymath}
\begin{diagram}
0&\rTo&C_b(P) &\rTo^{r_P} &\varprojlim_{1<Q\le
P}C_b(N_P(Q)/Q)&\rTo^{\tilde d_P}& \tilde H^0_b(\mathcal A_{\ge
2}(P), p\Zz)^P&\rTo&0
\end{diagram}
\end{displaymath}
of abelian groups.

When the group $P$ has rank $1$, we obtain the exact sequence
\begin{displaymath}
\begin{diagram}
0& \rTo & \Zz & \rTo & C_b(P) & \rTo^{r_P} & \varprojlim_{1<Q\le
P}C_b(N_P(Q)/Q)& \rTo& \Zz/m_P\Zz&\rTo&0
\end{diagram}
\end{displaymath}
of abelian groups where $m_P = 2$ if $P$ is cyclic and $m_P = 4$
if $P$ is generalized quaternion.

Finally, to simplify our notation, we introduce a general setup
for the gluing problem for an arbitrary biset functor in Section
\ref{sec:gluingGeneral}.

\section{Gluing problem in general}\label{sec:gluingGeneral}

Let $F$ be a biset functor and $G$ be a finite group. A
\emph{gluing data} for $F(G)$ is a sequence $(f_H)_{1<H\le G}$ of
elements $f_H\in F(N_G(H)/H)$ satisfying the following
compatibility conditions:
\begin{enumerate}
  \item[(i)] (Conjugation invariance) For any $g\in G$ and $H\le G$, we have $${}^gf_H = f_{{}^gH}.$$
  \item[(ii)] (Destriction invariance) For any pair $H\trianglelefteq K$ of subgroups of $G$, we have $$\des^{N_G(H)/H}_{N_G(H,K)/K}f_H = \res^{N_G(K)/K}_{N_G(H,K)/K}f_K. $$
\end{enumerate}
Here we use the notation $\des^{N_G(H)/H}_{N_G(H,K)/K}$ for the
composition of deflation and restriction maps. Note that the same
composition is written as Defres in \cite{Bouc-glue}. We call the
composite map $\des$ the \emph{destriction map}.

Now the \emph{gluing problem} for the biset functor $F$ at $G$ is
the problem of finding an element $f\in F(G)$ such that for any
non-trivial subgroup $H$ of $G$, we have
\[
\des^G_{N_G(H)/H}f = f_H.
\]
If such an element exists, we call it a \emph{solution} to the
given gluing data. Following Bouc and Th\'{e}venaz
\cite{BoucThe-glue}, we denote by $$\varprojlim_{1<H\le
G}F(N_G(H)/H)$$ the set of all gluing data for $F(G)$.

It is clear that given a biset functor $F$, there is a
homomorphism
\[
r_G:= r_G^F: F(G) \ra \varprojlim_{1<H\le G}F(N_G(H)/H)
\]
given by associating an element $f\in F(G)$ to the sequence
$(\des^G_{N_G(H)/H}f)_{1< H\le G}$. Now the gluing problem is
asking if this map is surjective, cf. \cite[Section
2]{BoucThe-glue}. A \emph{complete solution to the gluing problem}
is an exact sequence
\begin{displaymath}
\begin{diagram}
0&\rTo&\ker(r_G) &\rTo &F(G)&\rTo^{r_G}& \varprojlim_{1<H\le
G}F(N_G(H)/H)&\rTo&\obs(F(G))&\rTo&0
\end{diagram}
\end{displaymath}
where Obs$(F(G))$ denotes the cokernel of the map $r_G$ and it is
the group of obstructions for the gluing data to have a solution.

An easy (but very useful) observation about the general gluing
problem is that we can lift any short exact sequence of biset
functors to an exact sequence of the corresponding groups of
gluing data and hence relate the obstruction groups for the
solutions. This can be done as follows.

\begin{pro}\label{prop:commDiagGen} Let
\begin{displaymath}
\begin{diagram}
0&\rTo&M &\rTo &F&\rTo^\Phi& N&\rTo&0
\end{diagram}
\end{displaymath}
be an exact sequence of biset functors. Then the diagram
\begin{displaymath}
\begin{diagram}
0&\rTo&M(G) &\rTo &F(G)&\rTo^{\Phi^G}& N(G)&\rTo&0\\
 & &\dTo^{r_G^M} & & \dTo^{r_G^F}&&\dTo^{r_G^N}&&\\
0&\rTo&\varprojlim_{1<H\le G}M(W_G(H)) &\rTo &\varprojlim_{1<H\le
H}F(W_G(H))&\rTo^{\tilde\Phi^G}& \varprojlim_{1<H\le G}N(W_G(H))&&
\end{diagram}
\end{displaymath}
commutes and has exact rows. Here $\tilde\Phi^G$ is defined as the
sequence $(\Phi^{W_G(H)})_{1<H\le G}$ of maps and we put
$W_G(H):=N_G(H)/H$.
\end{pro}

The proof of the above proposition is trivial since $\Phi$ is a
morphism of biset functors. As a corollary, applying the snake
lemma, we get an exact sequence of abelian groups
\begin{displaymath}
\begin{diagram}
0&\rTo&K(r_G^M) &\rTo &K(r_G^F)&\rTo&
K(r_G^N)&\rTo&\obs(M(G))&\rTo&\obs(F(G))&\rTo &\obs(N(G))
\end{diagram}
\end{displaymath}
where we put $K(r_G^F):=\ker(r_G^F)$.

\begin{rem}\label{rem:destFunc}
Given a biset functor $F$, it is clear that the functor
associating a group $G$ to its gluing data $\varprojlim_{1<H\le
G}F(W_G(H))$ is not a biset functor. Nevertheless this functor has
a structure of a destriction functor. By a \emph{destriction
functor}, we mean a biset functor without induction and inflation
and in this sense, they are generalizations of restriction
functors from the context of Mackey functors, see \cite{Co-Alc}
for details. Indeed, the maps
\[
\res^G_K: \varprojlim_{1<H\le G}F(W_G(H))\ra \varprojlim_{1<H\le
K}F(W_K(H))
\]
for any pair of groups $K\le G$ are defined in \cite{BoucThe-glue}
and in a similar way, we can define deflation map and transport of
structure by a group isomorphism. We skip the details of the
definition of these maps.
\end{rem}

We end the section by the following generalization of a result of
\cite{BoucThe-glue}, which we will use later. It shows that the
gluing problem for an arbitrary biset functor $F$ at an elementary
abelian group is always solvable. We skip the proof since it is
almost identical to the proof of Lemma 2.2 of \cite{BoucThe-glue}.
\begin{lem}[Bouc-Th\'{e}venaz]\label{lem:elemAbel}
Let $E$ be an elementary abelian group. Then the map
\[
r_E^F: F(E) \ra \varprojlim_{1<F\le E}F(E/F)
\]
is surjective with a section given by
\[
(u_F)_{1<F\le E}\mapsto -\sum_{1<F\le E}\mu(1,F)\infl_{E/F}^E u_F.
\]
where $\mu$ is the Mobius function of the poset of subgroups of
$E$.
\end{lem}

\section{Gluing super class functions}\label{sec:gluingSuperClss}

Let $G$ be a finite group. We denote by $B(G)$ the Burnside ring
of $G$. We recall that the Burnside ring of a finite group $G$ is
the Grothendieck ring of the category of finite  $G$-sets and as a
free abelian group, it is generated by the set of conjugacy
classes of subgroups of $G$. By $B^*(G)$ we denote the $\Zz$-dual
of the Burnside ring. It is well-known that the group $B^*(G)$ can
be identified with the ring of super class functions $S(G)\ra
\Zz$. Here $S(G)$ is the set of all subgroups of $G$ and a
\emph{super class function} is a function which is constant on the
conjugacy classes of subgroups of $G$. The functor sending a group
$G$ to the dual Burnside ring $B^*(G)$ is a biset functor with the
usual actions of bisets. We refer to \cite{Bouc-yalcin} for
further details.

The following theorem shows that the answer to the the gluing
problem for super class functions is always affirmative.
\begin{thm}\label{thm:exaBurn} There is an exact sequence of abelian groups
\begin{displaymath}
\begin{diagram}
0&\rTo&\Zz &\rTo &B^*(G)&\rTo^{r_G}& \varprojlim_{1<H\le
G}B^*(N_G(H)/H)&\rTo&0
\end{diagram}
\end{displaymath}
where the map $r_G$ is as defined in the previous section.
\end{thm}
\begin{proof}
The proof has two parts. First we prove that the kernel of $r_G$
is isomorphic to $\Zz$ and second we show that the map $r_G$ is
surjective.

For the first part, let $f$ be a super class function with $r_G(f)
= 0$. By the definition of the map $r_G$, we have
\[
\des^G_{N_G(H)/H} f = 0
\]
for any non-trivial subgroup $H$ of $G$. But by the definition of
the action of destriction on $B^*$, we have
\[
0= (\des^G_{N_G(H)/H} f)(L/H) = f(L)
\]
for any $L/H\le N_G(H)/H$. Therefore the super class function $f$
is zero at all non-trivial subgroups of $G$. Also it is clear that
its value at the trivial subgroup can be chosen freely. Therefore
we have proved that a super class function $f$ is in the kernel of
$r_G$ if and only if it is zero at any non-trivial subgroup of
$G$, as required.

To prove that $r_G$ is surjective, let $(f_H)_{1<H\le G}$ be a
gluing data. We claim that the super class function $f\in B^*(G)$
defined by $f(H) = f_H(H/H)$ and $f(1) = 0$ is a pre-image of the
given gluing data. Firstly, $f$ is a super class function since
for any $H\le G$ and $g\in G$, we have
\[
f({}^xH) = f_{{}^xH}({}^xH/{}^xH) = ({}^xf_H)({}^xH/{}^xH) =
{}^x(f_H(H/H)) = f_H(H/H) = f(H).
\]
Here we use the conjugation invariance of the gluing data. Now it
remains to show that for any $1<H\le G$, we have
\[
\des^G_{N_G(H)/H} f = f_H.
\]
Notice that by the destriction invariance of the gluing data, for
any pair $H\trianglelefteq K$, we have
\[
f_H(K/H) = f_K(K/K).
\]
Indeed we have
\[
f_H(K/H) = (\des^{N_G(H)/H}_{N_G(H,K)/K}f_H)(K/K)=
(\res^{N_G(K)/K}_{N_G(H,K)/K}f_K)(K/K)=f_K(K/K).
\]
Now it follows that
\[
(\des^G_{N_G(H)/H} f)(K/H) = f(K) = f_K(K/K) = f_H(K/H),
\]
for any $K/H\le N_G(H)/H$, as required.
\end{proof}

\begin{cor}\label{cor:solSupClss} A gluing problem for super class functions at a finite group $G$ has infinitely many
solutions.
\end{cor}
\begin{proof}
Let $(f_H)_{1<H\le G}$ be a gluing data. Then the super class
function $f$ defined by putting $f(H) := f_H(H/H)$ if $H\neq 1$
and $f(G/1) = a$ is a solution of the gluing data for any $a\in
\Zz$.
\end{proof}

\section{Gluing rational class functions for $p$-groups}\label{sec:gluingRatclss}
Let $p$ be a prime number and $P$ be a finite $p$-group. Denote by
$\mathcal R_\Qq(P)$ the rational character ring of $P$. This is
the Grothendieck ring of the category of $\Qq P$-modules. With the
usual definitions of induction, restriction, inflation, deflation
and transport of structure, the assignment $P\mapsto \mathcal
R_\Qq(P)$ is a biset functor. We denote the $\Zz$-dual of this
biset functor by $\mathcal R^*_\Qq$. We call an element of
$\mathcal R_\Qq^*(P)$ a \emph{rational class function} for $P$. We
refer to \cite{Bouc-book} for further details.

Recall that by a theorem of Ritter and Segal, the linearization
map lin$: B(P)\ra \mathcal R_\Qq(P)$, associating a finite $P$-set
$X$ to the permutation $\Qq P$-module $\Qq X$ with basis $X$, is
surjective. Hence its transpose gives an injective map $\mathcal
R_\Qq^*(P) \ra B^*(P)$, cf \cite{Bouc-DadeBurnside}. Note also
that the linearization map is a morphism of biset functors.
Therefore the biset functor $\mathcal R_\Qq^*$ can be identified
with  a subfunctor of the dual $B^*$ of the Burnside functor $B$,
via the transpose of the linearization map.

Therefore this inclusion gives rise to an inclusion of the group
of gluing data for $\mathcal R_\Qq^*$ to that of $B^*$. Hence by
Corollary \ref{cor:solSupClss}, a gluing data for $\mathcal
R_\Qq^*$ can always be glued to a super class function. Thus we
need to determine when the solution is a rational class function.
We recall the following result from \cite[Lemma 4.2]{Bouc-yalcin}.
In the following, we denote the unique group of order $p$ by
$C_p$.

\begin{lem} [Bouc-Yalçın \cite{Bouc-yalcin}]\label{lem:buya-clss} Let $f$ be a super class function for the $p$-group
$P$. Then $f$ is a rational class function if and only if for any
subquotient $T/S$ of $P$, isomorphic to $C_p\times C_p$, the
condition
\[
f(P/S) - f(P/T) = \sum_{S < X < T}(f(P/X) - f(P/T))
\;\;\;\;\;\;\;\; (*).
\]
is satisfied.
\end{lem}

Now let $(f_Q)_{1<Q\le P}$ be a gluing data for $\mathcal
R_\Qq^*(P)$ and let $f$ be a solution of the data in $B^*(P)$.
Also let $T/S$ be a subquotient isomorphic to $C_p\times C_p$. We
claim that the equality
\[
f(P/Y) = f_S((N_P(S)/S)/(Y/S))
\]
is satisfied for any $S\unlhd Y\le P$. Indeed, by the definition
of $f$, we have
\[
f(P/Y) = f_Y((N_P(Y)/Y)/(Y/Y))
\]
and by the destriction invariance of the gluing data we have
\[
f_Y((N_P(Y)/Y)/(Y/Y)) = f_S((N_P(S)/S)/(Y/S)).
\]
Here we apply the condition with $S\unlhd Y$. Therefore we can
write the above equality $(*)$ as
\[
f_S(N_P(S)/S) - f_S(N_P(S)/S) =\sum_{S < X < T}(f_S(N_P(S)/X) -
f_S(N_P(S)/T)).
\]
Here we shorten $(N_P(S)/S)/(Y/S)$ to $N_P(S)/Y$. Now if $S$ is
not the trivial subgroup, this condition is satisfied since
$f_S\in \mathcal R_\Qq^*(N_P(S)/S)$. Thus we have proved the
following lemma.

\begin{lem}\label{lem:myClss} Let $(f_Q)_{1<Q\le P}$ be a gluing data for $\mathcal R_\Qq^*(P)$ and let $f$ be a solution
of the data in $B^*(P)$. Then $f$ is a rational class function if
and only if for any elementary abelian subgroup $E$ of $P$ of rank
2, the condition $(*)$ of Lemma \ref{lem:buya-clss} is satisfied.
\end{lem}

\label{subsec:fAt1} Our next aim is to determine the group of
obstructions for a gluing data for $\mathcal R_\Qq^*$. First note
that if $P$ has rank $1$, then the condition of the above Lemma is
trivially satisfied and hence we get that $\mathcal R_\Qq^* =
B^*$. Therefore we reduced to the previous case.

Therefore we assume that $P$ has rank at least $2$, hence $P$ is
neither cyclic nor generalized quaternion. Then by the previous
lemma, a solution $f$ of the gluing data $(f_Q)_{1<Q\le P}$ in the
dual Burnside group $B^*(P)$ is not a rational class function if
there is an elementary abelian subgroup $E$ of $P$ of rank 2 such
that the condition $(*)$ of Lemma \ref{lem:buya-clss} is not
satisfied. Notice that the condition $(*)$ actually determines the
value at the trivial group of the solution, since for $E$, the
condition $(*)$ can equivalently be written as
\[
f(P/1) = \sum_{1 < X < T}f(P/X) - pf(P/T)
\]
and the values $f(P/X)$ and $f(P/T)$ are determined by the gluing
data. Hence the obstruction group should check whether the value
at the trivial subgroup is well-defined, that is, if for all
possible choices of the subgroup $E$, the value $f(P/1)$ is
constant.

To determine the obstruction group, first we introduce a notation.
We denote by $\mathcal A_{\ge 2}(P)$ the $P$-poset of all
elementary abelian subgroups of $P$ of rank at least 2, and where
$P$ acts by conjugation. It is clear that any maximal elementary
abelian subgroup of $P$ of rank $2$ is an isolated vertex in
$\mathcal A_{\ge 2}(P)$. It is also easy to see that the rest of
the elementary abelian subgroups lie in a connected component,
which is called the \emph{big component} of $\mathcal A_{\ge
2}(P)$, see \cite[Lemma 2.1]{CaThClssTri}. We refer to Lemma 2.2
in \cite{CaThClssTri} and Lemma 3.1 in \cite{BoucThe-glue} for
further properties of this poset.

Now following Bouc and Th\'{e}venaz \cite{BoucThe-glue}, we write
$\tilde H^0(\mathcal A_{\ge 2}(P), \Zz)^P$ for the group of
$P$-invariant functions $f$ from $\mathcal A_{\ge 2}$ to the
additive group $\Zz$ of integers, satisfying the condition that
$f(E) = f(F)$ if $E\le F$, modulo the constant functions. By Lemma
3.1 in \cite{BoucThe-glue}, the group $\tilde H^0(\mathcal A_{\ge
2}(P), \Zz)^P$ is generated by the characteristic functions of the
classes of maximal elementary abelian subgroups of $P$ of rank 2.
Now we define a map, denoted by $\tilde d_P$, cf. \cite[Section
3]{BoucThe-glue},
\[
\tilde d_P: \varprojlim_{1<Q\le P}\mathcal R_\Qq^*(N_P(Q)/Q)\ra
\tilde H^0(\mathcal A_{\ge 2}(P), \Zz)^P
\]
by associating the data $(f_Q)_{1<Q\le P}$ to the class of the
function $\tilde f$ which maps a subgroup $E$ of order $p^2$ to
$\sum_{1 < X < E}f(P/X) -pf(P/E)$ and a subgroup $F$ of order more
than $p^2$ to $\tilde f(E)$ for some subgroup $E$ of $F$ of order
$p^2$.

Here we denote by $f$ the solution of $(f_Q)_{1<Q \le P}$ in
$B^*(P)$. Here after, we write $[\phi]$ to denote the image of
$\phi\in H^0(\mathcal A_{\ge 2}(P), \Zz)^P$ in the quotient group
$\tilde H^0(\mathcal A_{\ge 2}(P), \Zz)^P$.

\begin{lem}
The map
\[
\tilde d_P: \varprojlim_{1<Q\le P}\mathcal R_\Qq^*(N_P(Q)/Q)\ra
\tilde H^0(\mathcal A_{\ge 2}(P), \Zz)^P
\]
as defined above is a group homomorphism.
\end{lem}
\begin{proof}
First we need to check that the map is well-defined, that is, we
need to check that for any $(f_Q)_{1<Q\le P}\in
\varprojlim_{1<Q\le P}\mathcal R_\Qq^*(N_P(Q)/Q)$, the function
$\tilde f:=\tilde d_P((f_Q)_{1<Q\le P})$ is in $\tilde
H^0(\mathcal A_{\ge 2}(P), \Zz)^P$. First, it is clear, by
conjugation invariance of the gluing data, that the resulting
function is $P$-invariant. To prove the condition that $f(E) =
f(F)$ if $E\le F$, note that we only need to check this condition
for the big component of $\mathcal A_{\ge 2}(P)$, which contains a
normal elementary abelian $p$-subgroup. By Lemma 2.1 in
\cite{CaThClssTri}, this component is unique. So it suffices to
show that the function $\tilde f$ restricted to the set of
subgroups of order $p^2$ in the big component is constant. Let $E$
be a maximal element of the big component. Then by Lemma
\ref{lem:elemAbel}, the gluing data $\res^P_E((f_Q)_{1<Q\le P})$
has a solution which implies, by Lemma \ref{lem:myClss}, that for
any pair of subgroups $T,S$ of rank 2 of $E$, we have $\tilde f(T)
= \tilde f(S)$. Now the result follows since for any pair $T,S$ of
subgroups of rank 2 in the big component, $TS$ is elementary
abelian and contains both $T$ and $S$, as required. The claim that
$\tilde d_P$ is a group homomorphism is trivial.
\end{proof}

Unfortunately $\tilde d_P$ is not surjective but it is possible to
determine the image. First of all, if $\mathcal A_{\ge 2}(P)$ is
connected, then the group $\tilde H^0(\mathcal A_{\ge 2}(P),
\Zz)^P$ vanishes and hence the map $\tilde d_P$ becomes the zero
map. Thus, without loss of generality, we can assume that the
poset $\mathcal A_{\ge2}$ contains at least two connected
components.

Then by \cite[Lemma 3.1]{BoucThe-glue}, there is a unique central
subgroup $Z$ of $P$ of order $p$ contained in each maximal
elementary abelian subgroup of rank 2 and moreover given such a
group $E$, all the other subgroups of order $p$ are $P$-conjugate.

On the other hand, by \cite{Maz}, if $p$ is odd, and if the rank
of $P$ is at least 3, then by Proposition 2.5 in \cite{GlaMaz},
there  is a unique normal elementary abelian subgroup $E_0$ of $P$
of rank $2$ which is non-central and characteristic, which implies
that $E_0$ contains the subgroup $Z$ and all the other subgroups
of order $p$ are $P$-conjugate.

Similarly, if $p=2$ and the rank of $P$ is at least $3$, then
there is a non-maximal normal elementary abelian subgroup $E_0$ of
rank $2$ whose two non-central elements of order $2$ are
$P$-conjugate. (Note that in this case, as in the previous one, we
have $|P/C_P(E_0)| = 2$.)

Thus if $E$ is either an isolated vertex in $\mathcal A_{\ge
2}(P)$ or is equal to $E_0$, defined above, let $Y$ be a subgroup
of order $p$ in $E$ different from $Z$. Then we have
\[
\sum_{1 < X < E}f(P/X) - pf(P/E) = f(P/Z) + p(f(P/Y) - f(P/E))
\]
where we use the conjugation invariance of the gluing data to
collect together all the terms $f(P/X)$ with non-central $X$. Now
it is clear from this equation that
\[
[E\mapsto \sum_{1 < X < E}f(P/X) - pf(P/E)] = [E\mapsto p(f(P/Y) -
f(P/E))]
\]
for some non-central subgroup $Y$ of order $p$ in $E$. Therefore
the image of $\tilde d_P$ is contained in the subgroup $\tilde
H^0(\mathcal A_{\ge 2}(P), p\Zz)^P$ of $\tilde H^0(\mathcal A_{\ge
2}(P), \Zz)^P$ consisting of functions having values in $p\Zz$.

Next we show that this subgroup is equal to the image of $\tilde
d_P$. Note that the subgroup $\tilde H^0(\mathcal A_{\ge 2}(P),
p\Zz)^P$ is generated by functions $p\chi_E$ as $E$ runs over a
complete set of representatives of conjugacy classes of maximal
elementary abelian subgroups of rank 2 and where $\chi_E$ is the
characteristic function of the class of $E$.

Let $E$ be a maximal elementary abelian subgroup of $P$ of rank 2
and let $\chi_E$ be the characteristic function of the conjugacy
class of $E$. We define a sequence $(f_Q)_{1<Q\le P}$ as follows.
If $Q$ is not $P$-conjugate to $Y$, then put $f_Q = 0$. If $Q$ is
conjugate to $Y$, then define $f_Q$ as the super class function
for $N_P(Q)/Q$ given by $f_Q(N_P(Q)/T) = \delta_{T,Q}$, where
$\delta_{?,?}$ is the Kronecker's delta. We claim that the
sequence is a gluing data for $\mathcal R_\Qq^*(P)$ and its image
is equal to $p\chi_E$. The first claim is almost trivial since
most of the $f_Q$'s are zero and the group $N_P(Y)/Y$ is cyclic or
generalized quaternion by \cite[Lemma 3.1]{BoucThe-glue}. So we
need to check that $\tilde d_P((f_Q)_{1<Q\le P}) = p\chi_E$ but
this follows easily from the definition of the map $\tilde d_P$.

The above arguments prove that there is a surjective group
homomorphism
\[
\tilde d_P: \varprojlim_{1<Q\le P}\mathcal R_\Qq^*(N_P(Q)/Q)\ra
\tilde H^0(\mathcal A_{\ge 2}(P), p\Zz)^P.
\]
Next we show that the kernel of this map is exactly the image of
the map $r_P$. For this aim, let $f$ be a rational class function
and $(f_Q)_{1<Q\le P}$ be its image under $r_P$. Then
\[
\tilde d_P((f_Q)_{1<Q\le P})(E) = p(f(P/Y) - f(P/E))
\]
where $E$ and $Y$ are as above. But by Lemma \ref{lem:buya-clss},
we have $p(f(P/Y) - f(P/E)) = f(P/1) - f(P/Z)$ where $Z$ is as
before. Since $Z$ is contained in any maximal elementary abelian
subgroup of rank 2, the right hand side of the above equality is
independent of the choice of $E$, that is, $\tilde d_P(r_P(f))$ is
constant. Hence $\tilde d_P(r_P(f)) = 0$ as required. Similarly,
if $(f_Q)_{1<Q\le Q}$ is a sequence in the kernel of $\tilde d_P$,
then a solution of $(f_Q)_{1<Q\le P}$ in $B^*(P)$ can be chosen to
be a rational class function, by choosing the value at the trivial
subgroup to be the constant that is determined by the gluing data
under the map $\tilde d_P$ (so that the condition of Lemma
\ref{lem:myClss} is satisfied).

We remark that a similar argument shows that the map $r_P$ is
injective. Indeed a rational class function $f$ is in the kernel
of $r_P$ if and only if it is zero at any nontrivial subgroup of
$P$. But by Lemma \ref{lem:myClss}, the value at the trivial group
is uniquely determined by the values at nontrivial subgroups.
Hence $f$ should be the zero function, as required. With this
remark, we have finished the proof of the following theorem.

\begin{thm} \label{thm:glueRat} Let $p$ be a prime number and $P$ be a $p$-group of rank at least $2$. Then there is an
exact sequence of abelian groups
\begin{displaymath}
\begin{diagram}
0&\rTo&\mathcal R_\Qq^*(P) &\rTo^{r_P} &\varprojlim_{1<Q\le
P}\mathcal R_\Qq^*(N_P(Q)/Q)&\rTo^{\tilde d_P}& \tilde
H^0(\mathcal A_{\ge 2}(P), p\Zz)^P&\rTo&0
\end{diagram}
\end{displaymath}
where the maps $r_P$ and $\tilde d_P$ are as defined above.
\end{thm}

Note that when the poset $\mathcal A_{\ge 2}(P)$ is connected, the
group $\tilde H^0(\mathcal A_{\ge 2}(P), p\Zz)^P$ becomes the
trivial group and we get that the gluing problem always has a
unique solution, which proves the following corollary.
\begin{cor}
Let $P$ be a $p$-group of rank at least $2$ such that the poset
$\mathcal A_{\ge 2}(P)$ is connected. Then the map
\[
r_P^{\mathcal R_\Qq^*}: \mathcal R_\Qq^*(P) \ra
\varprojlim_{1<Q\le P}\mathcal R_\Qq^*(N_P(Q)/Q)
\]
is an isomorphism.
\end{cor}

Finally, the following corollary is an easy consequence of the
above proof, completing the solution of the gluing problem for
rational class functions for $p$-groups.
\begin{cor}
Let $P$ be a $p$-group of rank $1$. Then there is an exact
sequence
\begin{displaymath}
\begin{diagram}
0& \rTo & \Zz & \rTo & \mathcal R_\Qq^*(P) & \rTo^{r_P} &
\varprojlim_{1<Q\le P}\mathcal R_\Qq^*(N_P(Q)/Q)& \rTo&0
\end{diagram}
\end{displaymath}
of abelian groups.
\end{cor}

\section{Gluing Borel-Smith functions}\label{sec:gluingBS}

Let $p$ be a prime number and $P$ be a $p$-group. As in Section
\ref{sec:gluingSuperClss}, let $B^*(P)$ denotes the group of super
class functions for $P$. Following tom Dieck \cite{tomDie}, we
call a super class function $f\in B^*(P)$ a \emph{Borel-Smith
function} if the following conditions are satisfied:
\begin{enumerate}
  \item If $p$ is odd, then for any subquotient $T/S$ of $P$, of order $p$, the number $f(T) - f(S)$ is even.
  \item If $p=2$, then for any sequence of subgroups $H\unlhd K\unlhd L\le N_P(H)$, with $|K:H| = 2$, the value $f(K) - f(H)$ is even if $L/K$ is cyclic of order $4$ and
    is divisible by $4$ if $L/K$ is quaternion of order $8$.
  \item For any elementary abelian subquotient $T/S$ of $P$, of rank 2, the equality \[ f(S) - f(T) = \sum_{S<X<T} (f(X) -f(T)) \] holds.
\end{enumerate}
It is proved in \cite{Bouc-yalcin} that the subgroups $C_b(Q)$ of
Borel-Smith functions as $Q$ runs over all $p$-groups together
with the induced actions of bisets is a subfunctor of the biset
functor $B^*$.

To solve the gluing problem for Borel-Smith functions, we
introduce the following group. Define the group
$$\tilde H^0_b(\mathcal A_{\ge 2}(P), p\Zz)^P$$ as the subgroup of $\tilde H^0(\mathcal A_{\ge 2}(P), p\Zz)^P$ generated
by the elements $pn_E\chi_E$ where $E$ runs over a complete set of
representatives of conjugacy classes of maximal elementary abelian
subgroups of $P$ of rank $2$ and for such a subgroup $E$, the
integer $n_E$ is given by
\[
n_E= \left\{
       \begin{array}{ll}
         1, & \hbox{if $|C_P(Y):Y| = 2$;} \\
         2, & \hbox{if $C_P(Y)/Y$ is cyclic of order $\ge 3$;}\\
         4, & \hbox{if $C_P(Y)/Y$ is generalized quaternion}
       \end{array}
     \right.
\]
where $Y$ is a non-central proper subgroup of $E$ and recall that
by Lemma 2.2 in \cite{CaThClssTri}, the group $C_P(Y)/Y$ is either
cyclic or generalized quaternion.

Now we define a map
\[
\tilde d_P:  \varprojlim_{1<Q\le P}C_b(N_P(Q)/Q)\ra \tilde
H^0_b(\mathcal A_{\ge 2}(P), p\Zz)^P
\]
by associating a given sequence $(b_Q)_{1<Q\le P}$ to the class of
the function $E\mapsto p(b_S(W_P(S)/(S/S))-b_S(W_P(S)/(E/S)))$
where $S$ is a non-central proper subgroup of $E$, and $E$ is
chosen as in the case of the rational class functions. With this
notation, we can prove the following theorem.
\begin{thm}\label{thm:glueBS} Let $P$ be a $p$-group of rank at least $2$. There is an exact sequence of abelian groups
\begin{displaymath}
\begin{diagram}
0&\rTo&C_b(P) &\rTo^{r_P} &\varprojlim_{1<Q\le
P}C_b(N_P(Q)/Q)&\rTo^{\tilde d_P}& \tilde H^0_b(\mathcal A_{\ge
2}(P), p\Zz)^P&\rTo&0
\end{diagram}
\end{displaymath}
where the maps $r_P$ and $\tilde d_P$ are as defined above.
\end{thm}
\begin{proof}
The proof of the injectivity of $r_P$ is similar to the proof of
the injectivity of $r_P^{\mathcal R_\Qq^*}$. To prove that $\tilde
d_P$ is surjective, note that we can embed $\tilde H^0_b(\mathcal
A_{\ge 2}(P),p\Zz)^P$ into $\tilde H^0(\mathcal A_{\ge
2}(P),p\Zz)^P$. Now if $n_E\chi_E$ is the characteristic function
of the class of the maximal elementary abelian subgroup $E$ of $P$
of rank $2$, then we regard $pn_E\chi_E$ as a rational class
function and using the surjectivity of $\tilde d_P^{\mathcal
R_\Qq^*}$, we construct the pre-image $(f_Q)_{1<Q\le P}$ of
$n_E\chi_E$ in $\varprojlim_{1<Q\le P}\mathcal R_\Qq^*(N_P(Q)/Q)$.
Now this class is a gluing data for Borel-Smith functions since
$f_Q$ is non-zero only if $Q$ is $P$-conjugate to a non-central
proper subgroup of $E$ in which case it is either $1$ or $2$ or
$4$ according to whether $C_P(Q)/Q$ is cyclic of order $2$ or
cyclic of order at least $3$ or generalized quaternion and in this
case it is clear that the function $f_Q$ is a Borel-Smith
function, proving the surjectivity of the map $\tilde d_P$.

Finally, the proof of the exactness at the middle term of the
sequence is similar to the proof of the exactness at the middle
term of the sequence of Theorem \ref{thm:glueRat}.
\end{proof}

As in the previous case for rational class functions, if the poset
$\mathcal A_{\ge 2}(P)$ is connected, the group $H^0_b(\mathcal
A_{\ge 2}(P),p\Zz)^P$ vanishes and we obtain the following
corollary.
\begin{cor}
Let $P$ be a $p$-group such that the poset $\mathcal A_{\ge 2}(P)$
is connected. Then the map
\[
r_P^{C_b}: C_b(P) \ra \varprojlim_{1<Q\le P} C_b(N_P(Q)/Q)
\]
is an isomorphism.
\end{cor}

Finally, the following corollary completes the solution of the
gluing problem for Borel-Smith functions for $p$-groups. We leave
the straightforward proof as an exercise.
\begin{cor}
Let $P$ be a $p$-group of rank $1$. Then there is an exact
sequence
\begin{displaymath}
\begin{diagram}
0& \rTo & \Zz & \rTo & C_b(P) & \rTo^{r_P} & \varprojlim_{1<Q\le
P}C_b(N_P(Q)/Q)& \rTo& \Zz/m_P\Zz&\rTo&0
\end{diagram}
\end{displaymath}
of abelian groups where $m_P = 2$ if $P$ is cyclic and $m_P = 4$
if $P$ is generalized quaternion.
\end{cor}

\section{The Main Result}\label{sec:mainResults}

The main result of the paper is the following theorem which
relates the endo-trivial group to the group of obstructions of the
gluing problem for the functor of Borel-Smith functions. As an
immediate corollary, we will obtain a set of generators for the
group $T(P)$ of endo-trivial modules.

\begin{thm}\label{thm:trivExa} Let $p$ be a prime number and $P$ be a $p$-group, which is not cyclic, quaternion or
semi-dihedral. There is a split exact sequence of abelian groups
\begin{displaymath}
\begin{diagram}
0&\rTo&\Zz &\rTo &T(P)&\rTo& \obs(C_b(P))&\rTo&0
\end{diagram}
\end{displaymath}
where $T(P)$ is the group of endo-trivial modules and
$\obs(C_b(P))$ is the group of obstructions for the gluing problem
at $P$ for the biset functor of Borel-Smith functions.
\end{thm}
\begin{proof}
The proof consists of two parts. The first part is the proof of
the existence of the exact sequence and the second part is the
construction of an explicit splitting.

For the first part of the proof, recall that, by Theorem 1.2 of
\cite{Bouc-yalcin}, there is an exact sequence of $p$-biset
functors
\begin{displaymath}
\begin{diagram}
0&\rTo&C_b &\rTo &B^*&\rTo^{\Psi}& D^\Omega &\rTo&0.
\end{diagram}
\end{displaymath}
Here the map $\Psi$ is defined  as the family $(\Psi_P)$ where
$\Psi_P$ maps an element $\omega_{P/Q}$ of the basis $\{
\omega_{P/Q}| Q\le_P P\}$ of $B^*(P)$, introduced in Lemma 2.2 of
\cite{Bouc-DadeGroup}, to the corresponding relative syzygy
$\Omega_{P/Q}$.

Now by Proposition \ref{prop:commDiagGen}, there is a commutative
diagram

\begin{displaymath}
\begin{diagram}
0&\rTo&C_b(P) &\rTo &B^*(P)&\rTo^{\Psi_P}&D^\Omega(P)&\rTo&0\\
 & &\dTo^{r_P^{C_b}} & & \dTo^{r_P^{B^*}}&&\dTo^{r_P^{D}}&&\\
0&\rTo&\varprojlim_{1<Q\le P}C_b(W_P(Q)) &\rTo
&\varprojlim_{1<Q\le
P}B^*(W_P(Q))&\rTo^{\tilde\Psi_P}&\varprojlim_{1<Q\le P}
D^\Omega(W_P(Q))
\end{diagram}
\end{displaymath}
where the maps $\tilde\Psi_P$ and $r_P$'s are defined as in
Section \ref{sec:gluingGeneral}. Now by Snake Lemma together with
Theorem \ref{thm:exaBurn}, we obtain an exact sequence
\begin{displaymath}
\begin{diagram}
0&\rTo&Ker(r_P^{C_b}) &\rTo &\Zz&\rTo&
T^\Omega(P)&\rTo&\obs(C_b(P))&\rTo&0
\end{diagram}
\end{displaymath}
where $T^\Omega(P)$ denotes the kernel of $r_P^{D^\Omega}$. Now we
claim that the equality
\[
T^\Omega(P) = T(P)
\]
holds. Indeed, by \cite[Section 2.1.1]{Puig}, the kernel of
$r_P^D$ is equal to the group $T(P)$ and hence we have
\[
T^\Omega(P) = T(P)\cap D^\Omega(P).
\]
But by Theorem 7.7 in \cite{Bouc-DadeGroup}, any class in the Dade
group which is not equal to a class of relative syzygy has finite
order, and by Theorem 1.1 in \cite{CaThClssTorTri}, the group
$T(P)$ is torsion free if $P$ is not cyclic, generalized
quaternion or semi-dihedral. Thence, we have $T^\Omega(P) = T(P)$.

On the other hand, by the proof of Theorem \ref{thm:exaBurn}, the
group $\Zz$, in the above sequence, can be identified with the
subgroup of $B^*(P)$ generated by $\omega_P$. Now by the
definition of the map $\Psi$, the image of $\omega_P$ is equal to
$\Omega_P$. Moreover it is well-known that the subgroup of $T(P)$
generated by $\Omega_P$ is infinite cyclic since $P$ is not cyclic
or quaternion, see \cite{Dade-endo-2}, cf. \cite{Thevenaz-guided}.
Therefore we get that $\ker(r_P^{C_b}(P)) = 0$, which completes
the first part of the proof of the theorem.

For the second part of the proof, first note that the surjective
map $T(P)\ra \obs(C_b(P))$ is the connecting homomorphism of the
following diagram.
\begin{displaymath}
\begin{diagram}
&&0&\rTo& \Zz &\rTo&T(P)&&\\
 & &\dTo & & \dTo&&\dTo&&\\
0&\rTo&C_b(P) &\rTo &B^*(P)&\rTo^{\Psi_P}&D^\Omega(P)&\rTo&0\\
 & &\dTo^{r_P} & & \dTo^{r_P}&&\dTo^{r_P}&&\\
0&\rTo&\varprojlim_{1<Q\le P}C_b(W_P(Q)) &\rTo &\varprojlim_{1<Q\le P}B^*(W_P(Q))&\rTo^{\tilde\Psi_P}&\varprojlim_{1<Q\le P}D^\Omega(W_P(Q))\\
& &\dTo & & \dTo&&&&\\
&&\obs(C_b(P))&\rTo& 0& & & &
\end{diagram}
\end{displaymath}
To be able to construct a splitting for this connecting
homomorphism, notice that by Theorem \ref{thm:glueBS}, we also
have the following commutative diagram
\begin{displaymath}
\begin{diagram}
& & 0 & &0& & 0\\
 & &\dTo & & \dTo&&\dTo&&\\
0&\rTo&C_b(P) &\rTo &\mathcal R_\Qq^*(P)&\rTo^{\Psi_P}&D_t^\Omega(P)&\rTo&0\\
 & &\dTo & & \dTo&&\dTo&&\\
0&\rTo&\varprojlim_{1<Q\le P}C_b(W_P(Q)) &\rTo &\varprojlim_{1<Q\le P}\mathcal R_\Qq^*(W_P(Q))&\rTo&\varprojlim_{1<Q\le P}D_t^\Omega(W_P(Q))\\
 & &\dTo & & \dTo&&\dTo&&\\
0&\rTo&   \tilde H^0_b(\mathcal A_{\ge 2}(P), p\Zz)^P &\rTo&
\tilde
H^0(\mathcal A_{\ge 2}(P), p\Zz)^P&\rTo&\obs(D_t^\Omega)&&\\
 & &\dTo & & \dTo&&\dTo&&\\
& & 0 & &0& & 0\\
\end{diagram}
\end{displaymath}
and the identification
\[
\obs(C_b(P)) = \tilde H^0_b(\mathcal A_{\ge 2}(P), p\Zz)^P.
\]
Here, as usual, $\hookrightarrow$ denotes an injective map and
$\twoheadrightarrow$ denotes a surjective map.

Now the splitting is defined in the following way. First consider
the third diagram in this proof. Let $p n_E \chi_E\in \tilde
H^0_b(\mathcal A_{\ge 2}(P), p\Zz)^P$ be a generator of this group
where $E$ is a maximal elementary abelian subgroup $E$ of $P$ of
rank 2. We regard it as an element in $\tilde H^0(\mathcal A_{\ge
2}(P), p\Zz)^P$. By the proof of Theorem \ref{thm:glueRat}, the
pre-image of $pn_E\chi_E$ in $\varprojlim_{1<Q\le P}\mathcal
R_\Qq^*(W_P(Q))$ is the sequence $(f_Q)_{1<Q\le P}$ given by $f_Q
= 0$ unless $Q$ is conjugate to a non-central cyclic subgroup, say
$A$, of $E$ in which case $f_Q$ is given by $f_Q(T/Q) =
n_E\delta_{T,Q}$. Since it is not zero, the gluing data
$(f_Q)_{1<Q\le P}$ is not coming from a rational class function.

But we pass to the second diagram of this proof by embedding this
data to the group of the gluing data of super class functions.
Then by Theorem \ref{thm:exaBurn}, this data can be glued to a
super class function, say $f$, which is given by, $f(P/Q) = 0$
unless $Q$ is conjugate to $A$ in which case $f(P/Q) = n_E$. Using
the basis $\{ \omega_{P/Q}| Q\le_P P \}$ of $B^*(P)$, we can write
\[
f = n_E(\omega_{P/A} - \omega_{P/1}).
\]
Now by the definition of the map $\Psi^P$, we have
\[
\Psi^P(f) =
[(\Omega_{P/A}\otimes\Omega_{P/1}^{-1})^{\otimes^{n_E}}]
\]
where for an endo-permutation module $M$, we denote its class in
the Dade group by $[M]$. Hence the above construction gives a
homomorphism
\[
\sigma: H^0_b(\mathcal A_{\ge 2}(P), p\Zz)^P\ra D(P).
\]
of abelian groups. Now we claim that the class of
$(\Omega_{P/A}\otimes\Omega_{P/1}^{-1})^{\otimes^{n_E}}$ in the
Dade group $D(P)$ is in the image of the endo-trivial group.
Remark that this class coincides with the class found by Alperin
in \cite{Alperin-endo} where he proved this claim. However we can
give a direct proof of the claim.

To prove the claim, we need to prove that the class
$(\Omega_{P/A}\otimes\Omega_{P/1}^{-1})^{\otimes^{n_E}}$ is in the
kernel of the map $r_P^{D}$ but this is trivial since the gluing
data $(f_Q)_{1<Q\le P}$ is a data for Borel-Smith functions, by
its construction and hence the claim follows from the
commutativity of the second diagram of this proof.

Hence we get that the class
$[(\Omega_{P/A}\otimes\Omega_{P/1}^{-1})^{\otimes^{n_E}}]$ is in
the image of $T(P)$ in $D(P)$ which gives a homomorphism
\[
\sigma: H^0_b(\mathcal A_{\ge 2}(P), p\Zz)^P\ra T(P).
\]
of abelian groups. To end the proof, we need to check that this
map is a splitting for the connecting homomorphism but this is
straightforward from the construction.
\end{proof}
The following theorem follows from the above proof. Note that
although we end up with the same set of generators for the
endo-trivial group found by Alperin in \cite{Alperin-endo}, our
proof does not depend on his result, hence gives another proof of
the classification theorem of Carlson and Th\'{e}venaz in
\cite{CaThClssTri}.

\begin{thm}[Carlson-Th\'{e}venaz\cite{CaThClssTri}] Let $p>2$ be a prime number and $P$ be a non-abelian $p$-group which
is not semi-dihedral. Then the group $T(P)$ of endo-trivial
modules for $P$ is free abelian on the basis $\Omega_P,
\Gamma((\Omega_{P/A}\otimes \Omega_{P/1}^{-1})^{\otimes^{n_E}})$
as $E$ runs over a complete set of representatives of conjugacy
classes of maximal elementary abelian subgroups of $P$ of rank 2.
Here, for a $kP$-module $M$, the module $\Gamma(M)$ is the sum of
all indecomposable summands of $M$ with vertex $P$.
\end{thm}

We end the paper by noting that our proof does not use the
classification theorem of Carlson and Th\'{e}venaz, (which means
that our arguments are not circular). Indeed, we have used three
results from general theory of endo-permutation modules. Namely,
the theorems about two exact sequences from \cite{Bouc-yalcin},
Theorem 7.1 from \cite{Bouc-DadeGroup} which states that the
classes of endo-permutation modules which are not relative
syzygies are torsion and the injectivity of the map $r_P^{D_t}$
from \cite{BoucThe-glue}. Finally we have used the description of
the torsion part of the group $T(P)$ from \cite{CaThClssTorTri}. A
careful reading would show that all these results together with
any result that they refer are independent of this classification
theorem.

\textbf{Acknowledgement.} I would like to thank to Nadia Mazza for
answering a question about the poset $\mathcal A_{\ge 2}(P)$ and
Ergün Yalçın for his useful comments.

\end{document}